\documentclass[12pt,oneside]{amsart}
\usepackage{amscd,amsmath,amssymb,amsfonts,a4}
\usepackage[cmtip, all]{xy}
\usepackage{pstricks}
\usepackage{color}
\usepackage{hyperref}
\usepackage{enumitem}
\newlength{\rulebreite}

\def\timesover#1#2#3{\ \xymatrix@1@=0pt@M=0pt{ _{#1}&\times&_{#2} \\& ^{#3}&}\ }
\def\otimesover#1#2#3{\ \xymatrix@1@=0pt@M=0pt{ _{#1}&\otimes&_{#2} \\& ^{#3}&}\ }

\usepackage[all]{xy}
\theoremstyle{plain}

\newtheorem{thm}{Theorem}
\newtheorem{lem}[thm]{Lemma}
\newtheorem{cor}[thm]{Corollary}
\newtheorem{prop}[thm]{Proposition}
\theoremstyle{definition}
\newtheorem{defn}[thm]{Definition}

\newtheorem{question}[thm]{Question}
\newtheorem{rmk}[thm]{Remark}

\newtheorem{ex}[thm]{Example}

\numberwithin{thm}{section}
\numberwithin{equation}{section}

\newcommand{\sO}{{\mathcal O}}

\newcommand{\C}{{\mathbb C}}

\newcommand{\F}{{\mathbb F}}

\newcommand{\N}{{\mathbb N}}

\newcommand{\Q}{{\mathbb Q}}

\def\tilde{\widetilde}

\begin{document}

\title{Constraints on the cohomological correspondence associated to a self map }
\author{K. V. Shuddhodan}
\email{kvshud@math.tifr.res.in}
\address{School of Mathematics\\Tata Institute of Fundamental Research,\newline Homi
Bhabha Road\\Colaba\\Mumbai-400005\\India}
\curraddr{Freie Universit\"at Berlin\\Arnimallee 3\\14195\\Berlin\\Germany}

\begin{abstract}

In this article we establish some constraints on the eigenvalues for the action of a self map of a proper variety on its $\ell$-adic cohomology.~The essential ingredients are a trace formula due to Fujiwara,~and the theory of weights.

\end{abstract}

\maketitle

\section{Introduction}

Let $X$ be a smooth projective variety over a field $k$ and $f:X \to X$ a dominant self map of $X$.~When $k=\C$ one can look at the topological entropy of the dynamical system $(X(\C),f)$,~which has been shown to be equal to the logarithm of the spectral radius of $f^*$ acting on the Betti cohomology of $X(\C)$ (\cite{Gromov},~\cite{Yomdin}).~When $k$ is an arbitrary field,~there is no obvious and useful notion of a topological entropy,~however it still makes sense to look at the action of $f^*$ on suitable cohomology theories.

Esnault-Srinivas studied this question for automorphisms of surfaces over finite fields (\cite{Esnault-Srinivas}). They look at the action of $f^*$ on the $\ell$-adic cohomology of a smooth proper surface defined over an algebraic closure of a finite field.~In this case they show that the spectral radius on $H^{*}_{\text{\'et}}(X,\Q_{\ell})$ coincides with the spectral radius on the smallest $f^*$ stable sub-algebra generated by any ample class $\omega \in H^{2}_{\text{\'et}}(X,\Q_{\ell})$.~Since such a sub-space of the cohomology is stable under specialization,~this at once proves the result for automorphisms of surfaces over arbitrary base fields.~This in particular leads naturally to the question of whether the spectral radius restricted to the even degree cohomology is the same as that on the entire cohomology (see \cite{Esnault-Srinivas},~Section 6.3).

From an arithmetical viewpoint there is an abundance of self maps arising from the Galois group of the base field.~The corresponding action on cohomology gives rise to interesting Galois representations.~These Galois representations are often hard to understand and contain information about the variety itself.~We combine these arithmetical and dynamical points of view to define a \textit{zeta function} associated to a self map of a variety over a finite field (see Definition \ref{Zeta}).~We obtain the principal result of this article by studying the analytical properties of this zeta function,~which we shall state now.

Let $k$ be either an algebraic closure of a finite field or the field of complex numbers $\C$.~Fix a prime $\ell$ invertible in $k$ (if $\text{char}(k)>0$),~and an embedding
 
  \begin{equation}\label{l-adic_complex_embedding_intro}
  \tau: \Q_{\ell} \hookrightarrow \C.
  \end{equation}

Suppose $X$ is a proper scheme over $k$.~Let $H^*(X)$ be the $\ell$-adic cohomology of $X$ (when $\text{char}(k)>0$) with its increasing weight filtration (\cite{Deligne-Weil II},~Section 5.3.6), or the singular cohomology $H^*(X(\C),\Q)$ (when $k=\C$) with its mixed Hodge structure (\cite{Deligne-Hodge III},~Proposition 8.1.20).~Let $W_kH^*(X)$ be the associated weight filtration.
 
 Let $f:X \to X$ be a self map of $X$.
  
 Let $\lambda_{\text{odd}}$ (resp. $\lambda_{\text{even}}$) be the spectral radius (with respect to $\tau$ if $\text{char}(k)>0$) for the action of $f^*$ on $\oplus_{i \geq 0} H^{2i+1}(X)$ (resp. $\oplus_{i \geq 0} H^{2i}(X)$).~Let $k_{\text{odd}}$ be the maximum among integers $k$ with the property that the spectral radius for the action of $f^*$ on $\text{Gr}^{k}_WH^i(X)$ is $\lambda_{\text{odd}}$,~where $i$ is an odd integer.~Similarly define $k_{\text{even}}$.
   
 \begin{thm}\label{entropy_restriction_intro}
 
  Using the above notations,~we have
 
  \begin{enumerate}
  
 \item[(1)] $\lambda_{\text{even}} \geq \lambda_{\text{odd}}$.
  
  \item[(2)] If equality holds in $(1)$,~then $k_{\text{even}} \geq k_{\text{odd}}$.
  
  \end{enumerate}
  
 \end{thm}

The inequality in $(2)$ seems to be previously unknown even for smooth projective varieties over $\C$.~The inequality in $(1)$  extends what was previously known for smooth projective varieties over $\C$ using analytical techniques (via the Gromov-Yomdin theory) to arbitrary proper varieties over $\C$.~However the properness assumption is necessary,~as shown by Example \ref{theorem_false_non_proper}.

Theorem \ref{entropy_restriction_intro} immediately allows us to deduce the following Corollary.

\begin{cor}\label{spectral_radius_arbitrary_base_field_intro}

Let $f:X \to X$ be a self map of a proper scheme over an arbitrary field $k$.~Let $\ell$ be a prime invertible in $k$ and $\bar{k}$ an algebraic closure of $k$.~Fix an embedding $\tau: \Q_{\ell} \hookrightarrow \C$.~Then the spectral radius (with respect to $\tau$) for the action of $f_{\bar{k}}^*$ on the entire $\ell$-adic cohomology $H^*(X_{\bar{k}},\Q_{\ell})$ is equal to the spectral radius for its action on $\oplus_{ i \geq 0} H^{2i}(X_{\bar{k}},\Q_{\ell})$.

\end{cor}

These constraints on the eigenvalues can be seen as a consequence of a positivity condition.~When the base field is finite,~and the variety is proper,~the logarithm of the zeta function (see Definition \ref{Zeta}) is a power series with positive coefficients and it is this positivity that we exploit (via Lemma \ref{convergence}).~The geometric Frobenius plays a dual role here,~on the one hand it ensures the above positivity (see Proposition \ref{Fujiwara_trace_fixed_point}),~and on the other hand we can `distinguish' cohomology classes using their weights.

In a recent work,~Truong (\cite{Tuyen})  among other things has shown that for a dominant self map of a smooth projective variety over an arbitrary base field the spectral radius on the $\ell$-adic cohomology is bounded above by its spectral radius on the ring of algebraic cycles mod numerical equivalence,~thus implying that the spectral radius on the even degree cohomology and the entire cohomology are the same.~We believe that the point of view taken in this article is different from the one in \cite{Tuyen}.

\subsubsection*{Acknowledgements:}

I would like to thank my advisor Prof.~Vasudevan Srinivas for his support and generosity.~I am grateful to Prof.~H\'el\`ene Esnault for reading an earlier draft of this article,~and suggesting corrections which have helped improve the overall presentation.~I would like to thank Prof.~Najmuddin Fakhruddin for enlightening comments and questions.~I would also like to thank the referee for several helpful comments and suggestions.~This work has been partly supported by SPM-07/858(0177)/2013/EMR-1.

\section{A Zeta function associated to a self map}\label{Zeta_Function_self_map}

In this section we shall introduce a \textit{zeta function} associated to a proper self map of a separated scheme of finite type over $\F_q$.~The analytic properties of this zeta function will be central to the results in Section \ref{new_constraints_eigenvalues}.

Let $\F_q$ be a finite field with $q$ elements of characteristic $p$.~Let $k$ be an algebraic closure of $\F_q$.~Let $\ell$ be a prime co-prime to $p$.

Schemes over $\F_q$ will be denoted by a sub-script $0$ (for example $X_0,~Y_0$,~etc.).~Similarly morphisms of schemes over $\F_q$ will be denoted by $f_0,~g_0$,~etc.~The corresponding object over $k$ will be denoted without a sub-script,~for example $X,~f$,~etc. .

Let $X$ be a separated scheme of finite type over $k$.~Let $H^{i}(X,\Q_{\ell})$ and $H^{i}_c(X,\Q_{\ell})$ respectively denote the $i^{\mathrm{th}}$ usual and compactly supported $\ell$-adic \'etale cohomology of $X$.

For any self map $f: X \to X$,~we define

\begin{equation}\label{trace_usual}
\text{Tr}(f^*,H^*(X,\Q_{\ell})):=\sum_{i=0}^{2 \text{dim}(X)} (-1)^{i} \text{Tr}(f^*;H^{i}(X,\Q_{\ell})) \in \Q_{\ell}.
\end{equation}

Similarly when $f$ is proper,~we define

\begin{equation}\label{trace_compactly_supported}
\text{Tr}(f^*,H^*_c(X,\Q_{\ell})):=\sum_{i=0}^{2 \text{dim}(X)} (-1)^{i} \text{Tr}(f^*;H^{i}_c(X,\Q_{\ell})) \in \Q_{\ell}.
\end{equation}

\begin{defn}[Absolute Frobenius]\label{Absolute_Frobenius}

For any scheme $X_0/\F_q$,~the \textit{absolute Frobenius} (with respect to $\F_q$) $F_{X_0,q}:X_0 \to X_0$ is the morphism which is the identity on the underlying topological space,~and for any open affine sub-scheme $U_0 \subseteq X_0$,~$F_{X_0,q}|_{U_0}$ corresponds to the ring endomorphism of $\Gamma(U_0,\sO_{U_0})$ given by $\alpha \to \alpha^q$.

\end{defn}

\begin{rmk}

The absolute Frobenius is a morphism over $\F_q$.

\end{rmk}

The following lemma is standard,~and we state it without a proof.

\begin{lem}\label{unramified_Frobenius}

Let $X_0$ be separated scheme of finite type over $\F_q$.~Then $X_0/\F_q$ is unramified (and hence \'etale) iff $F_{X_0,q}$ is an unramified morphism.

\end{lem}

\begin{defn}[Geometric Frobenius]\label{Geometric_Frobenius}

For a scheme $X_0/\F_q$ let $X$ denote the base change of $X_0$ to $k$.~The \textit{geometric Frobenius} (with respect to $\F_q$) is the morphism $F_{X,q}:X \to X$ induced from $F_{X_0,q}:X_0 \to X_0$ by base change.

\end{defn}

\begin{rmk}

The geometric Frobenius morphism is a morphism over $k$.

\end{rmk}

A scheme $X/k$ is said to be \textit{defined over $\F_q$},~if there exists a scheme $X_0/\F_q$ and an isomorphism of $k$-schemes between $X$ and $X_0 \times_{\F_q} k$.~Given any such scheme $X/k$ defined over $\F_q$ and a choice of a $\F_q$ structure as above,~the geometric Frobenius morphism $F_{X_0 \times_{\F_q} k}$ induces a self map of $X/k$.~We call this the \textit{geometric Frobenius} (with respect to $\F_q$ and $X_0$) and denote it by $F_{X,q}$. 

In a similar vein,~a diagram of schemes over $k$ is said to be defined over $\F_q$ if it is obtained as a base change of a diagram over $\F_q$.

Now suppose $f_0:X_0 \to X_0$ is a self map of a separated scheme of finite type over $\F_q$.~Let $(X,f)$ be the corresponding pair over $k$,~obtained by base change.
Let $\Gamma^t_{f_0 \circ F_{X_0,q}}$ be the transpose of the graph of $f_0 \circ F_{X_0,q}$ and $\Delta_{X_0}$ denote the diagonal embedding of $X_0$ in $X_0 \times_{\F_q} X_0$.~Note that $f_0$ and $F_{X_0,q}$ commute.

\begin{lem}\label{etale}

With notations as above,~the scheme $\Gamma^t_{f_0 \circ (F_{X_0,q})^m} \cap \Delta_{X_0}$ is  \'etale over $\F_q$ for all $m \geq 1$.

\end{lem}

\begin{proof}

If $m$ is greater than $1$,~replacing $f_0$ by $f_0 \circ (F_{X_0,q})^{m-1}$ we reduce to the case $m=1$.

Let $Z_0:=\Gamma^{t}_{f_0 \circ F_{X_0,q}} \cap \Delta_{X_0}$ and $g_0:=f_0 \circ F_{X_0,q}$.~We have a cartesian diagram 

\begin{center}

$\xymatrix{ Z_0 \ar@{^{(}->}[d]^{i}  \ar@{^{(}->}[r]^{i} & X_0 \ar@{^{(}->}[d]^{\Delta} &  \\
                   X_0 \ar@{^{(}->}[r]^{\Gamma^t_{g_0}} \ar[dr]_{g_0} & X_0 \times_{\F_q} X_0 \ar[r]^{\text{pr}_2} \ar[d]^{\text{pr}_1} & X_0 \\
                                                  & X_0                            & }$.

\end{center}

Here the morphism $\Gamma^t_{g_0}$ is the transpose of the graph of $g_0$ and $\text{pr}_i,~i=1,2$ are the projections.~The commutativity of the above diagram implies that,~$(f_0 \circ i) \circ F_{Z_0,q}=(f_0 \circ F_{X_0,q}) \circ i$,~is a closed immersion.~Hence $F_{Z_0,q}$ is a closed immersion and thus an unramified morphism.~Lemma \ref{unramified_Frobenius} now implies that $Z_0$ is unramfied (and hence \'etale) over $\F_q$.

\end{proof}

Let $f:X \to X$ be a self map of a finite type scheme over $k$.

\begin{defn}[Fixed point]\label{Fixed_point}

A closed point $x \in X$ is said to be a \textit{fixed point} of $f$ if $f(x)=x$.

\end{defn}

\begin{prop}\label{positivity_trace_smooth_proper}

Let $f_0:X_0 \to X_0$ be a self map of a smooth,~proper scheme over $\F_q$.~Then $\text{Tr}((f^m \circ F_{X,q}^n)^* ,H^*(X,\Q_{\ell}))$ is the number of fixed points of $f^m \circ F_{X,q}^n$ acting on $X$.

\end{prop}

\begin{proof}

This is an immediate consequence of Lemma \ref{etale} and the trace formula in \cite{SGA 4.5},~Chapitre 4,~Corollaire 3.7.

\end{proof}

The above proposition naturally leads to the following questions,~for an arbitrary $X_0$,~a proper $f_0$ and integers $m \geq 1$.

\begin{question}\label{Question_Fixed_point}

Is the $\ell$-adic number $\text{Tr}\left ((f \circ F_{X,q}^{m})^{*},H^*_{c}(X,\Q_{\ell}) \right )$ an integer? If yes,~then is it equal to the number of fixed points of $f \circ F_{X,q}^{m}$ acting on $X$?

\end{question}

A trace formula by Fujiwara sheds some light on these questions.

\begin{prop}\label{Fujiwara_trace_fixed_point}

Both the questions in \ref{Question_Fixed_point} have a positive answer,~if $m$ is sufficiently large.~Moreover when $X_0$ is proper any $m \geq 1$ would do.

\end{prop}

\begin{proof}

This is an immediate consequence of Fujiwara's trace formula (see \cite{Fujiwara},~Corollary 5.4.5,~compare \cite{Illusie},~Section 3.5 (b)).

\end{proof}

Proposition \ref{Fujiwara_trace_fixed_point} motivates the following definition.

\begin{defn}\label{least_twist}

Let $n_0(f)$ be the least integer $m$ such that,~both the questions in \ref{Question_Fixed_point} have a positive answer.

\end{defn}

Now we are in a position to define the zeta function,~and study its analytical properties.

Let $X_0$ be a separated scheme of finite type over $\F_q$,~and $f_0:X_0 \to X_0$ a proper self map (also defined over $\F_q$).~Let $\ell$ be a prime invertible in $\F_q$.

For any field $K$,~let $K[[t]]$ and $K((t))$ respectively be the ring of formal power series and the field of formal Laurent series with coefficients in $K$.

\begin{defn}\label{Zeta}
The \textit{zeta function} $Z(X_0,f_0,t)$ corresponding to $(X_0,f_0)$ is defined to be 

\begin{center}

 $Z(X_0,f_0,t)=\exp(\sum_{n \geq 1}\dfrac{\text{Tr}((f \circ F_{X,q})^{n*},H^*_c(X,\Q_{\ell}))t^n}{n}) \in \Q_{\ell}[[t]] \subset \Q_{\ell}((t))$.

\end{center}

\end{defn}

\begin{rmk}

When $f_0$ is the identity,~one recovers the Hasse-Weil zeta function of the scheme $X_0$. 

\end{rmk}

\begin{lem}\label{Zeta_independent_l}

$Z(X_0,f_0,t)$ belongs to $\Q(t) \cap \Q_{\ell}[[t]] \subseteq \Q_{\ell}((t))$ and is independent of $\ell$.

\end{lem}

\begin{proof}

The result in \cite{Illusie},~Section 3.5 (b) implies that the above power series is actually in $\Q[[t]]$ and is independent of $\ell$.~The trace-determinant relation (\cite{Deligne-Weil I},~1.5.3) implies that $Z(X_0,f_0,t)$ is a rational function in $\Q_{\ell}$.~Hence $Z(X_0,f_0,t) \in \Q_{\ell}(t) \cap \Q[[t]] \subset \Q_{\ell}[[t]]$ and hence in $\Q(t)$ (see \cite {Deligne-Weil I},~Lemme 1.7).~Thus $Z(X_0,f_0,t) \in \Q(t)$ and is independent of $\ell$.
 \end{proof}
 
 \begin{cor}\label{radius_convergence}
 
 The formal power series $Z(X_0,f_0,t) \in \Q[[t]] \subseteq \C[[t]]$ has a non-trivial radius of convergence about the origin in the complex plane and has a meromorphic continuation onto the entire complex plane as a rational function.
 
 \end{cor}

\begin{proof}

It follows from Lemma \ref{Zeta_independent_l} that,~the formal power series $Z(X_0,f_0,t)$ coincides with the power series expansion (about the origin) of a rational function with coefficients in $\Q$,~as elements of $\Q[[t]]$.~Hence the result.

\end{proof}

\begin{rmk}\label{Zeta_constant_map}
Suppose $X_0$ is geometrically connected and $f_0:X_0 \to X_0$ a constant map (necessarily mapping to a $\F_q$-valued point).~The associated zeta function $Z(X_0,f_0,t)$ is $\frac{1}{1-t} \in \Q[[t]]$.~We shall see in the next section that even when $f_0$ is not dominant,~the zeta function associated to it still carries enough information for us to use.~This is in contrast to the techniques in \cite{Esnault-Srinivas},~where $f_0$ being an automorphism was crucially used.
\end{rmk}

\section{Some new constraints on the eigenvalues}\label{new_constraints_eigenvalues} 

In this section we establish some new constraints on the eigenvalues for the action of a self map of a proper variety on its $\ell$-adic cohomology.

We continue using the notations and conventions of Section \ref{Zeta_Function_self_map}.
 
 \begin{lem}\label{convergence}
 
 Let $G(t) \in t\Q[[t]]$ be a formal power series with non-negative coefficients and with no constant term.~Then the formal power series $\exp(G(t)) \in \Q[[t]]$ and $G(t)$ have the same radius of convergence about the origin in the complex plane.~In particular,~in the disc of its convergence,~the formal power series $\exp(G(t))$,~considered as a holomorphic function,~coincides with the exponential (in the analytic sense) of a holomorphic function.
 
 \end{lem}
 
 \begin{proof}
 
It is clear that the radius of convergence of the formal power series $\exp(G(t))$ is at least as large as that of the formal power series $G(t)$,~subject to the latter having a non-trivial radius of convergence.~To complete the proof it suffices to show that the radius of convergence of the formal power series $\exp(G(t))$ is bounded above by the radius of convergence of $G(t)$. 
 
Using the comparison test for convergence it suffices to show that,

\begin{center}
$\frac{d^n}{dt^n}(\exp(G(t)))|_{t=0} \geq G^{(n)}(0)$,
\end{center}

\noindent where $G^{(n)}(t)$ is the $n^{\rm th}$ formal derivative of $G(t)$ (\cite{Rudin_Real_Complex_Analysis},~10.6,~(8)).~We shall show by induction on $n$ that,

\begin{center}

$\frac{d^n}{dt^n}(\exp(G(t))=P_n(G^{(1)}(t),G^{(2)}(t),~\cdots ,G^{(n)}(t)) \exp(G(t))$ in $\Q[[t]]$,

\end{center}

\noindent where $P_n(x_1,x_2,\cdots,x_n)$ is a polynomial with \textit{positive integral} coefficients and is of the form $P_n(x_1,x_2,\cdots,x_n)=x_n+\tilde{P}_{n}(x_1,x_2,\cdots,x_{n-1})$ for some polynomial $\tilde{P}_{n}$ in one less variable.
 
 For $n=1$ the statement is obviously true.~Assume now that the statement is true with $n=k$,~the chain rule of differentiation then implies the statement for $n=k+1$.~In particular one observes that
 
 \begin{center}
 
 $\frac{d^n}{dt^n}(\exp(G(t)))|_{t=0}=G^{(n)}(0)+\tilde{P}_{n}(G^{(1)}(0),G^{(2)}(0),\cdots,G^{(n-1)}(0))$.

 \end{center}

The non-negativity of the coefficients of $G(t)$ implies the non-negativity of $G^{(n)}(0)$.~Since the coefficients of the polynomial $\tilde{P}_{n}$ are positive we are done.
 
 \end{proof}
 
 \begin{lem}\label{exponential_rational_function}
Let $R(t),Q(t) \in \C[t]$ be polynomials with complex coefficients.~Let $G(t) \in t\Q[[t]]$ be a formal power series with non-negative coefficients and with no constant term such that \begin{equation}\label{exponential_rational_function_1}
\exp(G(t))=\frac{R(t)}{Q(t)},
\end{equation}
as elements in $\C[[t]]$.

Then any closed disc around the origin in the complex plane containing a root of $R(t)$ necessarily contains a root of $Q(t)$.
 \end{lem}
 
 \begin{proof}
First note that Equation (\ref{exponential_rational_function_1}) implies that $\exp(G(t))$ has a non trivial radius of convergence.~Let $\Delta$ be a closed disc around the origin not containing a root of $Q(t)$.~Then the equality in (\ref{exponential_rational_function_1}) implies that $\exp(G(t))$ is convergent in an open neighbourhood of $\Delta$.

Moreover since $G(t)$ has non-negative coefficients,~Lemma \ref{convergence} and Equation (\ref{exponential_rational_function_1}) together imply that the rational function $\frac{R(t)}{Q(t)}$ equals the exponential of a holomorphic function on $\Delta$,~and hence $R(t)$ cannot have a root in $\Delta$.
  
 \end{proof}
 
Let $k$ be either an algebraic closure of a finite field or the field of complex numbers $\C$.~Let $X/k$ be a proper scheme.

Fix a prime $\ell$ invertible in $k$ (if $\text{char}(k)>0$),~and an embedding
 
  \begin{equation}\label{l-adic_complex_embedding}
  \tau: \Q_{\ell} \hookrightarrow \C.
  \end{equation}

 Let $f:X \to X$ be a self map of $X$ (over $k$).~Let $\lambda_{even},~\lambda_{\text{odd}}$,~$k_{\text{even}}$,~$k_{\text{odd}}$ be as defined in the introduction.~Then one has the following result.
  
 \begin{thm}\label{entropy_restriction}

  \begin{enumerate}
  
 \item[(1)] $\lambda_{\text{even}} \geq \lambda_{\text{odd}}$.
  
  \item[(2)] If equality holds in $(1)$,~then $k_{\text{even}} \geq k_{\text{odd}}$.
  
  \end{enumerate}
  
 \end{thm}
 
 \begin{proof}
 
Using a standard spreading out argument we can reduce to the case when $k$ is an algebraic closure of a finite field.~Moreover we fix a model $(X_0,f_0)$ for the pair $(X,f)$ over a finite sub field $\F_q$ of $k$.

Note that for any positive integer $r$,~the spectral radius (with respect to $\tau$) of $f^{r*}$ acting on the odd and even degree cohomology are $\lambda^r_{\text{odd}}$ and $\lambda^r_{\text{even}}$ respectively.~There exists an iterate of $f$ which maps a connected component of $X$ into itself.~Hence $f^*$ acting on $H^0(X,\Q_{\ell})$ has at least one eigenvalue of modulus $1$ (with respect to any $\tau$).~Thus $\lambda_{\text{even}} \geq 1$.~If $\lambda_{\text{odd}}=0$ we are done,~hence we can assume that $\lambda_{\text{odd}} \neq 0$.

Fix a positive integer $r$ and consider the zeta function (see Definition \ref{Zeta}) of the pair $(X_0,f_0^r)$.~It follows from Corollary \ref{radius_convergence} that this zeta function is defined as a holomorphic function in a non-trivial neighbourhood of the origin,~and has a  meromorphic continuation onto the entire complex plane as a rational function.~Suppose that $\frac{R(t)}{Q(t)}$ is the meromorphic continuation,~where $R(t)$ and $Q(t)$ are co-prime rational polynomials in $t$.~Moreover using the trace-determinant relation (see \cite{Deligne-Weil I},~1.5.3),~one has
 
 \begin{equation}\label{Zeta_trace_determinant}
 Z(X_0,f_0^r,t)=\prod_{i=0}^{2\text{dim}(X)} P_{i}(t)^{(-1)^{i+1}} \text{in} \hspace{0.2cm} \Q_{\ell}[[t]] \subset \C[[t]] (\text{via} \hspace{0.2cm} \tau) 
 \end{equation}
 
 where $P_i(t)=\text{det} \left(1-t \left(F_{X,q} \circ f^r \right)^{*},H^{i} \left( X, \Q_{\ell} \right) \right),~0 \leq i \leq 2 \, \text{dim}(X)$.\\

Moreover one also has

\begin{equation}\label{Zeta_meromorphic}
Z(X_0,f_0^r,t)=\frac{R(t)}{Q(t)} \hspace{0.2cm} \text{in} \hspace{0.2cm} \Q[[t]].
\end{equation}

Hence (\ref{Zeta_trace_determinant}) and (\ref{Zeta_meromorphic}) imply that the complex roots of $R(t)$ and $Q(t)$ are a subset of the \textit{inverse eigenvalues} of $(F_{X,q} \circ f^r)^*$ acting on the odd and even degree cohomology,~respectively.~In particular,~they are non-zero.~Also note that $f^*$ and $F_{X,q}^*$ commute and hence can be simultaneously brought to a Jordan canonical form.~Hence any eigenvalue of $F_{X,q}^* \circ f^{r*}$ acting on any $H^i(X,\Q_{\ell})$ is a product of an eigenvalue of $F_{X,q}^*$ and one of $f^{r*}$ acting on the same $H^i(X,\Q_{\ell})$. 

Let $\alpha$ be any complex root of $\prod_{0 \leq i \leq \text{dim}(X)} P_{2i}(t) \in \C[t]$ (via $\tau$ in (\ref{l-adic_complex_embedding})).~Since the $i^{\rm th}$ compactly supported $\ell$-adic cohomology is of weight less than or equal to $i$ (\cite{Deligne-Weil II},~Th\'eor\`eme 1 (3.3.1)),~one has,

\begin{equation}\label{even_root_bound}
\frac{1}{|\alpha|} \leq \lambda^r_{\text{even}}q^{\text{dim}(X)}.
\end{equation}

In particular,~\em $Q(t)$ has no roots on the closed disc of radius $\frac{1}{q^{\text{dim}(X)}\lambda^r_{\text{even}}}$ \em.

Since $\lambda^r_{\text{odd}}$ is the spectral radius (with respect to $\tau$) for $f^{r*}$ acting on the oddly graded cohomology,~there exists an odd index $2i+1$ and a complex root $\beta$ of $P_{2i+1}(t)$ such that,

\begin{equation}\label{odd_root_bound}
|\beta| = \frac{1}{\lambda_{\text{odd}}^rq^{m(\beta)}} \leq  \frac{1}{\lambda_{\text{odd}}^r},
\end{equation}

\noindent where $2m(\beta)$ is a non-negative integer (corresponding to the weight of the Frobenius). In particular, $\beta$ is not a root of $\prod_{0 \leq i \leq \text{dim}(X)} P_{2i}(t) \in \C[t]$. Hence $\beta$ is a root of $R(t)$.~Thus \em $R(t)$ has a root in the closed disc of radius $\frac{1}{\lambda^r_{\text{odd}}}$ \em.

It follows from Proposition \ref{Fujiwara_trace_fixed_point} that the zeta function $Z(X_0,f_0^k,t)$ is of the form $\exp(G(t))$, where $G(t) \in t\Q[[t]]$ has non-negative coefficients.~Thus Lemma \ref{exponential_rational_function} implies that

\begin{center}
$\frac{1}{\lambda_{\text{odd}}^{r}} \geq \frac{1}{q^{\text{dim}(X)}\lambda^r_{\text{even}}}$,
\end{center}

\noindent and hence

\begin{center}
$\lambda^{r}_{\text{even}}q^{\text{dim}(X)} \geq \lambda_{\text{odd}}^{r}$.
\end{center}
 
Since $r$ was an arbitrary positive integer this implies 

\begin{center}
$\lambda_{\text{even}} \geq \lambda_{\text{odd}}$.
\end{center}
 
Now suppose $\lambda_{\text{even}}=\lambda_{\text{odd}}$.~We shall prove that $k_{\text{even}} \geq k_{\text{odd}}$.

Let $\mu_i$ be the spectral radius (with respect to $\tau$ in (\ref{l-adic_complex_embedding})) for the action of  $f^*$ on $H^{2i}(X,\Q_{\ell})$ for each $0 \leq i \leq \text{dim}(X)$.~Then there exists an integer $r >> 0$ such that any integer $i \in [0,~\text{dim}(X)]$ with $\mu_i \neq \lambda_{\text{even}}$,~we have

\begin{equation}\label{basic_inequality_weights}
\mu_i^r q^{\text{dim}(X)} < \lambda^r_{\text{even}} \leq \lambda^r_{\text{even}} q^{\frac{k_{\text{even}}}{2}}.
\end{equation}

As before Lemma \ref{Zeta_independent_l} and Corollary \ref{radius_convergence} imply that the zeta function $Z(X_0,f_0^r,t)$ has a non-trivial radius of convergence about the origin and has a meromorphic continuation of the form $R(t)/Q(t)$ onto the entire complex plane,~with $R(t)$ and $Q(t)$ being co-prime rational polynomials.~Also the zeroes of $R(t)$ and $Q(t)$,~are a subset of the inverse eigenvalues of $(F_{X,q} \circ f^r)^*$ acting on the odd and even degree cohomology respectively.

Since the $i^{\rm th}$ compactly supported $\ell$-adic cohomology is of weight less than or equal to $i$ (\cite{Deligne-Weil II},~Th\'eor\`eme 1 (3.3.1)),~the weight filtration on $H^i(X,\Q_{\ell})$ satisfies $W_kH^i(X,\Q_{\ell})=H^i(X,\Q_{\ell})$ for $k \geq i$.

Let 

\begin{center}

$P_{i,k}(t):=\text{det}(1-t(F_{X,q} \circ f^r)^{*},\text{Gr}^{k}_WH^{i}(X, \Q_{\ell})),~0 \leq k \leq i,~0 \leq i \leq 2 \, \text{dim}(X)$.

\end{center}

As before,~

\begin{center}

$P_i(t):=\text{det}(1-t(F_{X,q} \circ f^r)^{*},H^{i}(X, \Q_{\ell})),~0 \leq i \leq 2 \, \text{dim}(X)$.

\end{center}

Since the weight filtration is respected by the action of $F_{X,q}^* \circ f^{r*}$ (see \cite{Deligne-Weil II},~Section 5.3.6),~one has an equality, 

\begin{center}

$\prod_{i=0}^{\text{dim}(X)} \prod_{k=0}^{2i} P_{2i,k}(t)=\prod_{0 \leq i \leq \text{dim}(X)} P_{2i}(t) \in \C[t]$ (via $\tau$ in (\ref{l-adic_complex_embedding})).

\end{center}

Let $\alpha$ be any complex zero of $\prod_{0 \leq i \leq \text{dim}(X)} P_{2i}(t) \in \C[t]$.~Then $\alpha$ is a zero of $P_{2i,k}(t)$ for some integer $i \in [0,\text{dim}(X)]$ and $0 \leq k \leq 2i$.~If $\mu_i \neq \lambda_{\text{even}}$,~then (\ref{basic_inequality_weights}) implies,~

\begin{equation}
\frac{1}{|\alpha|} \leq \mu_i^r q^{\frac{k}{2}} \leq \mu_i^r q^{\text{dim}(X)}.
\end{equation}

On the other hand the definition of $k_{\text{even}}$ implies that,~if $\mu_i=\lambda_{\text{even}}$,~then $k \leq k_{\text{even}}$.~Thus  (\ref{basic_inequality_weights}) implies that for any complex root $\alpha$ of  $\prod_{0 \leq i \leq \text{dim}(X)} P_{2i}(t)$,

\begin{equation}
\frac{1}{|\alpha|} \leq \lambda_{\text{even}}^r q^{\frac{k}{2}} \leq \lambda_{\text{even}}^r q^{\frac{k_{\text{even}}}{2}}.
\end{equation}

In particular the same holds true for any complex root of $Q(t)$.

Moreover from the definition of $k_{\text{odd}}$ it follows that there exist an odd index $2i+1$ and a complex root $\beta$ of $P_{2i+1,k_{\text{odd}}}(t)$ such that,~$\frac{1}{|\beta|}=\lambda_{\text{odd}}^r q^{\frac{k_{\text{odd}}}{2}}=\lambda_{\text{even}}^r q^{\frac{k_{\text{odd}}}{2}}$.~If $k_{\text{odd}} \neq k_{\text{even}}$,~$\beta$ is not a root of $\prod_{0 \leq i \leq \text{dim}(X)} P_{2i}(t)$,~and thus is necessarily a root of $R(t)$.~Thus $R(t)$ has a root on the closed disc of radius $\frac{1}{\lambda_{\text{odd}}^r q^{\frac{k_{\text{odd}}}{2}}}$.~Arguing as before using Proposition \ref{Fujiwara_trace_fixed_point} and Lemma \ref{exponential_rational_function} we conclude that,

\begin{center}
$\frac{1}{\lambda_{\text{odd}}^r q^{\frac{k_{\text{odd}}}{2}}} \geq \frac{1}{ \lambda_{\text{even}}^r q^{\frac{k_{\text{even}}}{2}}}$.
\end{center}

Hence $k_{\text{even}} \geq k_{\text{odd}}$.

\end{proof}

\begin{cor}\label{spectral_radius_arbitrary_base_field}

Let $f:X \to X$ be a self map of a proper scheme over an arbitrary field $k$.~Let $\ell$ be a prime invertible in $k$ and $\bar{k}$ an algebraic closure of $k$.~Fix an embedding $\tau: \Q_{\ell} \hookrightarrow \C$.~Then the spectral radius (with respect to $\tau$) for the action of $f_{\bar{k}}^*$ on the entire $\ell$-adic cohomology $H^*(X_{\bar{k}},\Q_{\ell})$ is equal to the spectral radius for its action on $\oplus_{ i \geq 0} H^{2i}(X_{\bar{k}},\Q_{\ell})$.

\end{cor}

The following simple example shows that the inequality $(1)$ in Theorem \ref{entropy_restriction} need not be strict.

\begin{ex}\label{entropy_equality_odd_even}

Let $k=\C$ and $X=E \times E \times E$ where $E$ is any elliptic curve over $\C$. 

Let $f$ be the automorphism of $X$ given by $f(x,y,z)=(2x+3y,x+2y,z)$.~We have $f=g \times 1_{E}$,~where $g:E \to E$ is the automorphism $(x,y) \mapsto (2x+3y,x+2y)$,~and $1_{E}$ is the identity map on $E$.~Let $\lambda_{\text{even}},~\lambda_{\text{odd}},~k_{\text{even}},~k_{\text{odd}}$ be as defined in the previous section for the action of $f^*$ on $H^*(X(\C),\Q)$.~One can easily show that $\lambda_{\text{even}}=\lambda_{\text{odd}}=(2+\sqrt{3})^2$,~while $k_{\text{even}}=4$ and $k_{\text{odd}}=3$.

\end{ex}

The following example shows that Theorem \ref{entropy_restriction} is false without the properness hypothesis,~even for smooth varieties.

\begin{ex}\label{theorem_false_non_proper}

Let $T/k$ be a rank $2$ torus (split over $\F_q$) and $f:T \to T$ any group automorphism of $T$.~Then,~$\text{ker}(f \circ F_{T,q}^n-1_T)$ is a finite \'etale group scheme.~Since it is a sub-group scheme of a torus,~its order (or rank) is co-prime to $p$,~the characteristic of $\F_q$.~For any integer $n \in \N$,~let $\text{Fix}(f \circ F^n_{T,q})$ be the number of fixed points of $f \circ F^n_{T,q}$ acting on $T$.~Then 

\begin{center}
$\text{Fix}(f \circ F^n_{T,q})=|P_f(q^n)|$
\end{center}

 where $P_f(t):=\text{det}(1-tM_f;X^*(T))$,~and $M_f$ is the linear map on the co-character lattice $X^*(T)$,~induced by $f$.

Let $\ell$ be a prime different from $p$.~The only non-trivial compactly supported $\ell$-adic cohomology groups of $T$ are $H^2_c(T,\Q_{\ell}) \simeq \Q_{\ell}$,~$H^3_c(T,\Q_{\ell}) \simeq X^{*}(T) \otimes_{\Q_{\ell}} \Q_{\ell}(-1)$ and $H^4_c(T,\Q_{\ell}) \simeq \bigwedge^{2} X^*(T) \otimes_{\Q_{\ell}} \Q_{\ell}(-2)$.~Thus for all integers $n \geq 1$,

\begin{center}
$\text{Tr}((f \circ F^n_{T,q})^*,H^*_c(T,\Q_{\ell}))=P_f(q^n)$.
\end{center}

Let $f$ be chosen such that $|\text{Tr}(M_f)|>2$ and $\text{det}(M_f)=1$.~Note that the eigenvalues of $f^*$ acting on the compactly supported $\ell$-adic cohomology are algebraic integers independent of $\ell$.~Since $\text{Tr}((f \circ F^n_{T,q})^*,H^*_c(T,\Q_{\ell}))=P_f(q^n)$,~it follows that,~for any embedding $\tau: \Q_{\ell} \hookrightarrow \C$,~

\begin{enumerate}

\item[(1)] $f^*$ acting on $H^3_c(T,\Q_{\ell})$ has at least one eigenvalue of modulus greater than $1$.

\item[(2)] $f^*$ acts as identity on $H^2_c(T,\Q_{\ell})$ and $H^4_c(T,\Q_{\ell})$.

\end{enumerate}

Hence $\lambda_{\text{odd}} > \lambda_{\text{even}}$.

\end{ex}

A careful look at the proof of Theorem \ref{entropy_restriction} shows that the failure of Theorem \ref{entropy_restriction} without the properness hypothesis (as in Example \ref{theorem_false_non_proper}),~can be attributed to $n_{0}(f^{r})$ (see Definition \ref{least_twist}) being strictly greater than $1$.~This motivates us to consider the growth of an upper bound for $n_{0}(f^r)$ with respect to $r$ for non-proper varieties.~In a future work it will be shown that there exists an upper bound which grows at most linearly with $r$.

\end{document}